\documentclass[11pt,reqno]{amsart}
\usepackage{amscd,amssymb,amsmath,amsthm}
\usepackage[arrow,matrix]{xy}
\usepackage{graphicx}
\usepackage{cite}
\newtheorem{thm}[subsection]{Theorem}

\newtheorem{pro}[subsection]{Proposition}
\newtheorem{cor}[subsection]{Corollary}

\newtheorem{rk}[subsection]{Remark}
\newtheorem{defn}[subsection]{Definition}
\newtheorem{ex}{Example}

\numberwithin{equation}{section} 

\newcommand{\bea}{\begin{eqnarray}}
\newcommand{\eea}{\end{eqnarray}}

\begin{document}

\title[Periodic algebras generated by groups]{Periodic algebras generated by groups}

\author{S. Albeverio, B.A. Omirov, U. A. Rozikov}

 \address{S.\ Albeverio\\ Institut f\"ur Angewandte Mathematik and HCM, Universit\"at Bonn, Bonn, Germany.}
\email {albeverio@uni-bonn.de}

 \address{B.A.Omirov and U.\ A.\ Rozikov\\ Institute of mathematics and information technologies,
29, Do'rmon Yo'li str., 100125, Tashkent, Uzbekistan.}
\email {omirovb@mail.ru \ \ rozikovu@yandex.ru}

\begin{abstract}

We consider algebras with basis numerated by elements of a group $G.$ We fix a function $f$ from $G\times G$ to a ground field and give a multiplication of the algebra which depends on $f$. We study the basic properties of such algebras. In particular, we find a condition on $f$ under which the corresponding algebra is a Leibniz algebra. Moreover, for a given subgroup $\widehat G$ of $G$ we define a $\widehat G$-periodic algebra, which corresponds to a $\widehat G$-periodic function $f,$ we establish a criterion for the right nilpotency of a $\widehat G$-periodic algebra. In addition, for $G=\mathbb Z$
we describe all $2\mathbb Z$- and $3\mathbb Z$-periodic algebras. Some properties of $n\mathbb Z$-periodic algebras are obtained.
\end{abstract}

\maketitle

\section{Introduction}

Infinite dimensional algebras were introduced in mathematics at the beginning of the last century, they have had a considerable development during the past 40 years, from affine Lie algebras and loop groups to quantum groups in their various flavors. They also have found an ever increasing variety of applications in many domains of physics, from various aspects of solid state physics to most sophisticated models of quantum field theory, see, e.g., \cite{A2}, \cite{AF}, \cite{KM}, \cite{Schm}.

In the survey article \cite{Z} the author discusses some old and some new open questions on infinite-dimensional algebras. The paper describes interactions between combinatorial group theory, Lie algebras and infinite-dimensional associative algebras.  In  \cite{BO} a tame filtration of an algebra is defined by the growth of its terms, which has to be majorated by an exponential function.  The notion of tame filtration is useful in the study of possible distortion of degrees of elements when one algebra is embedded as a subalgebra in another algebra.  These authors consider  the case of associative or Lie algebras in the case of tame filtration of an algebra can be induced from the degree filtration of a larger algebra.

Whereas that the theory of finite dimensional algebras is well developed in a systematic way, it is fair to say that this is not get the case for the theory of infinite dimensional algebras (see, however, e.g., \cite{BO}, \cite{K},\cite{KM}, \cite{Z} and references there in).  In this paper we consider algebras over a field $K$, with  basis set $\{e_a, \ a\in G\}$, which is indexed by elements of a  group $(G,\circ)$. The multiplication table is given as $e_ae_b = f(a,b)e_{a\circ b\circ t}$, where $t$ is a fixed element of $G$ and $f$ a map of a Cartesian square of $G$
into the field $K.$ A construction of this type for an $n$-ary algebra was first considered in \cite{p1}. For an infinite group $G$ our construction gives an infinite dimensional algebra.

For a given subgroup $\widehat G$ of $G$  we define a notion of $\widehat G$-periodic algebra and study some its basic properties. Let us point out that corresponding notions of $\widehat G$-periodic Gibbs measures and periodic $p-$harmonic functions were considered in \cite{R1},\cite{R2},\cite{R3}, respectively. Whereas, the notion of $\widehat G$-periodic algebra can be considered in arbitrary variety of algebras.

In the present work we limit ourselves to the study a particular case of $\widehat G$-periodic algebras. Namely, we study $\widehat G$- periodic Leibniz algebras with additive group $G$.

It is known that Leibniz algebras are generalization of Lie algebras \cite{Lo2}. A lot of papers has been devoted to the description of finite dimensional Leibniz algebras. In the study of a variety of algebras the classification of algebras in low dimensions plays a crucial role.  Moreover, some conjectures can be verified in low dimensions. In the past much work has been invested into the classification of various varieties of algebras over the field of the complex numbers and fields of positive characteristics. We recall that the description of finite dimensional complex Lie algebras has been reduced to the classification of nilpotent Lie algebras, which have been completely classified up to dimension 7 (see \cite{M}, \cite{S}). In the case of Leibniz algebras the problem of classification of complex Leibniz algebras has been solved up to dimension 4 \cite{A1}, \cite{A3}. However, the classification of infinite dimensional algebras is more complicated. In this paper we give a construction of certain infinite dimensional algebras which are relatively simple to describe.

In our case from the Leibniz identity, we derive the functional equation for $f$. Thus, the problem of the classification of corresponding Leibniz algebras can be
reduced to the problem of the description of the functions $f$ up to a
non-degenerate basis transformation. Moreover in periodic cases our construction reduces the study of infinite dimensional algebras to the study of finite dimensional matrices.

\section{Preliminaries}

Let $A$ be an arbitrary algebra and let $\{e_1, e_2, \dots\}$ be a basis of the algebra $A$. The table of multiplication on the algebra is defined
by the products of the basic elements, namely,  $e_ie_j=\sum_{k}\gamma_{i,j}^ke_k,$
where $\gamma_{i,j}^k$ are the structural constants.

We recall that Leibniz algebras are defined by the Leibniz identity:
$$[x,[y,z]]=[[x,y],z]-[[x,z],y].$$

If the identity $[x,x]=0$ holds in Leibniz algebras then the Leibniz
identity coincides with the Jacobi identity. Thus, the Leibniz
algebras are a ``non commutative" analogue of the Lie algebras.

Let $L$ be a Leibniz algebra, we define the lower central sequence and, respectively, the derived sequence by:
$$L^1=L,\quad L^{n+1}=[L^n,L], \  \ n \geq 1,$$
\qquad \qquad \qquad \qquad resp.  \ $L^{[0]}=L, \ \ \ L^{[k+1]}=[L^{[k]},L^{[k]}], \ \  k \geq 1.$

\begin{defn} An algebra $L$ is called {\it nilpotent} ({\it solvable}) if there exits some $s \in {\mathbb{N}}$ (respectively, $t \in \mathbb{N}$) such that $L^s=0$ (respectively, $L^{[t]}=0$).
The minimal such number $s$ (resp. $t$) is called index of nilpotency (resp. solvability).
\end{defn}

For any element $x$ of $L$ we define  the operator $R_x$ of right multiplication as follows
$$\begin{array}{llll}
R_x:&z&\rightarrow & [z,x], \ z\in L
\end{array}$$

\begin{defn} Let $G$ be an additive group. An algebra $A$ is called $G$-{\it graded algebra} if $A$ can be decomposed into the direct sum of vector spaces (i.e. $A=\sum_{g\in G}\oplus A_g$) and $A_gA_h\subset A_{g+h}$, $\forall g,h\in G$.
\end{defn}

\section{An algebra generated by a group}

Let $K$ be a field. Fix an element $t$ of a group $(G,\circ)$ and consider the algebra $A_G(f,t)=\langle e_a | a\in G\rangle$ given by the following multiplication:
\begin{equation}\label{f}
e_a e_b=f(a,b)e_{a\circ b\circ t},
\end{equation}
 where $f : G\times G \longrightarrow K$ is a given function.

Denote by $\mathbf 1$ the unit element of the group $G$.
 Note that the set $\{e_a: a\in G\}$ is a group with binary operation $e_a\ast e_b=e_{a\circ b}$.

 \begin{pro}\label{pb} Take $t\in \bigcap_{{H\leq G:\atop H\ne \{\mathbf{1}\}}}H$.
 \begin{itemize}
\item[1)]  If  $H$ and $J$ are subgroups of the group $G$ and $f$ is such that $f(h,t^{-1})\ne 0$, for any $h\in H$. Then $A_H(f,t)$ is a subalgebra of $A_J(f,t)$ iff $H$ is a subgroup of $J$.

\item[2)] If $M\subset G$ with $M\circ M\circ t\subseteq M$ then  $A_M(f,t)$ is a subalgebra $A_G(f,t)$, but the opposite is not true, in general.

\item[3)] Let  $H$ be a subgroup of $G$ and $f$ is such that $f(h,g)\ne 0$ for all $h\in H$, $g\in G$. Then $A_H(f,t)$ is an ideal of $A_G(f,t)$ iff $H=G$.
    \end{itemize}
  \end{pro}
  \begin{proof} 1) If $H\leq J$ then $\{e_h\,| \, h\in H\}\subset \{e_j\,|\, j\in J\}$, consequently $A_H(f,t)$ is a subalgebra of $A_J(f,t)$. Now assume $A_H(f,t)$ is a subalgebra of $A_J(f,t)$ then for any $h\in H$ we consider
  $$e_he_{t^{-1}}=f(h,t^{-1})e_h.$$ Since $f(h,t^{-1})\ne 0$, we have $e_h\in \{e_j\,|\, j\in J\}$. Hence
  $h\in J$.

  2) Straightforward.

  3) If $H=G$ then $A_H(f,t)= A_G(f,t)$. If $A_H(f,t)$ is an ideal of $A_G(f,t),$ then
  for any $e_h\in A_H(f,t)$ and $e_g\in A_G(f,t)$ we have
  $$e_he_g=f(h,g)e_{h\circ g\circ t}\in A_H(f,t).$$
          This gives $h\circ g\circ t\in H$, which is equivalent to $g\in H$, for all $g\in G$.
  \end{proof}

 \begin{pro}\label{ps} Let $G$ be a commutative group. Then the algebra $A_G(f,t)$  is solvable of index $m$ if and only if there exists $m\in \mathbb N$ such that
 \begin{equation}\label{F1}
 \prod_{k=0}^{m-1}\ \ \prod_{s=0}^{2^{m-k-1}-1}f\left(t^{2^k-1}\circ\prod_{q=1}^{2^k}\circ a_{2^{k+1}s+q},\ \ t^{2^k-1}\circ\prod_{l=1}^{2^k}\circ a_{2^{k+1}s+2^k+l}\right)=0,
 \end{equation}
 for any $a_1,\dots, a_{2^m}\in G$.
  \end{pro}
  \begin{proof} First we shall prove the following formula 
  
  $$(\dots(e_{a_1}e_{a_2})\dots(e_{a_{2^r-1}}e_{a_{2^r}})\dots)=$$
  \begin{equation}\label{F2} 
  \prod_{k=0}^{r-1} \ \prod_{s=0}^{2^{r-k-1}-1}f\left(t^{2^k-1}\circ\prod_{q=1}^{2^k}\circ a_{2^{k+1}s+q},\ \ t^{2^k-1}\circ\prod_{l=1}^{2^k}\circ a_{2^{k+1}s+2^k+l}\right)e_{t^{2^r-1}\circ\prod_{i=1}^{2^r}\circ a_i}.
  \end{equation}
  We use mathematical induction by $r$. For $r=1$ and $r=2$ using the equality (\ref{f}) we obtain
  $$r=1: \ \ e_{a_1}e_{a_2}=f(a_1,a_2)e_{a_1\circ a_2\circ t}.$$
  $$r=2: \ \ (e_{a_1}e_{a_2})(e_{a_3}e_{a_4})=f(a_1,a_2)e_{a_1\circ a_2\circ t}f(a_3,a_4)e_{a_3\circ a_4\circ t}=$$
  $$ f(a_1,a_2)f(a_3,a_4)f(t\circ a_1\circ a_2, t\circ a_3\circ a_4)e_{a_1\circ a_2\circ a_3\circ a_4\circ t^3}.$$
  Thus the equality (\ref{F2}) is true for $r=1$ and $r=2$. Assume that (\ref{F2}) is true for $r$ and prove it for $r+1$.
  By the assumption of the induction we get
  $$(\dots(e_{a_1}e_{a_2})\dots(e_{a_{2^r-1}}e_{a_{2^r}})\dots)(\dots(e_{a_{2^r+1}}e_{a_{2^r+2}})
  \dots(e_{a_{2^{r+1}-1}}e_{a_{2^{r+1}}})\dots)=U\times V\times W,$$
  where
  $$U=\prod_{k=0}^{r-1} \prod_{s=0}^{2^{r-k-1}-1}f\left(t^{2^k-1}\circ\prod_{q=1}^{2^k}\circ a_{2^{k+1}s+q},\ \ t^{2^k-1}\circ\prod_{l=1}^{2^k}\circ a_{2^{k+1}s+2^k+l}\right),$$
 $$V=\prod_{k=0}^{r-1} \prod_{s=0}^{2^{r-k-1}-1}f\left(t^{2^k-1}\circ\prod_{q=1}^{2^k}\circ a_{2^r+2^{k+1}s+q},\ \ t^{2^k-1}\circ\prod_{l=1}^{2^k}\circ a_{2^r+2^{k+1}s+2^k+l}\right),$$
 $$W=e_{t^{2^r-1}\circ\prod_{i=1}^{2^r}\circ a_i}e_{t^{2^r-1}\circ\prod_{i=1}^{2^r}\circ a_{2^r+i}}=$$
$$f\left(t^{2^r-1}\circ\prod_{i=1}^{2^r}\circ a_i, \, t^{2^r-1}\circ\prod_{i=1}^{2^r}\circ a_{2^r+i}\right)e_{t^{2^{r+1}-1}\circ\prod_{i=1}^{2^{r+1}}\circ a_{i}}.$$

In $V$ we change $s$ with $s=s'-2^{r-k-1}$ then
$$V=\prod_{k=0}^{r-1} \prod_{s'=2^{r-k-1}}^{2^{r-k}-1}f\left(t^{2^k-1}\circ\prod_{q=1}^{2^k}\circ a_{2^{k+1}s'+q},\ \ t^{2^k-1}\circ\prod_{l=1}^{2^k}\circ a_{2^{k+1}s'+2^k+l}\right).$$
Consequently, we get
$$U\times V=\prod_{k=0}^{r-1} \prod_{s=0}^{2^{r-k}-1}f\left(t^{2^k-1}\circ\prod_{q=1}^{2^k}\circ a_{2^{k+1}s+q},\ \ t^{2^k-1}\circ\prod_{l=1}^{2^k}\circ a_{2^{k+1}s+2^k+l}\right).$$
Hence we have
$$U\times V\times W=\prod_{k=0}^{r} \prod_{s=0}^{2^{r-k}-1}f\left(t^{2^k-1}\circ\prod_{q=1}^{2^k}\circ a_{2^{k+1}s+q},\ \ t^{2^k-1}\circ\prod_{l=1}^{2^k}\circ a_{2^{k+1}s+2^k+l}\right)e_{t^{2^{r+1}-1}\circ\prod_{i=1}^{2^{r+1}}\circ a_{i}}.$$
This gives (\ref{F2}) for $r+1$.
By formula (\ref{F2}) one can see that the algebra $A_G(f,t)$  is solvable of index $m$ if and only if the condition (\ref{F1}) is satisfied. \end{proof}
Denote by $A_g$ the vector space $K e_g$. Then we have $A_G(f,0)=\sum_{g\in G}\oplus A_g$. Since $A_gA_h=\{\alpha e_g \beta e_h=\alpha \beta f(g,h)e_{g+h}\,|\, \alpha,\beta\in K,\, g,h\in G\}$ we get $A_gA_h\subseteq A_{g+h}$. Hence the algebra $A_G(f,0)$ is a graded algebra.

Let $G$ be an additive group. We shall find a condition on $f$ under which the algebra $A_G(f,t)$ will be a Leibniz algebra. From the Leibniz identity we get for the function $f$ the following equation
\begin{equation}\label{e1}
f(b,c)f(a,b+c+t)=f(a,b)f(a+b+t,c)-f(a,c)f(a+c+t,b).
\end{equation}

For a given $t\in G$ we set
 $$F_t=\{f \ : \  \mbox{for the given} \ t, \ f\ \mbox{is a solution of} \ (\ref{e1})\}.$$

Denote by $L(f,t)$ the algebra which is given by $f\in F_t$.

\begin{pro}\label{p1} For any $t, t'\in G$ and $f\in F_t$ there exists $g\in F_{t'}$ such that $L(f, t)\cong L(g, t')$ (where $\cong$ means algebraically isomorph).
\end{pro}
\begin{proof} Take the isomorphism $\varphi(e_i)=e'_{i+t-t'}$. Then
$$e'_ae'_b=e_{a-(t-t')}e_{b-(t-t')}=f(a-(t-t'), b-(t-t'))e_{a+b-2(t-t')+t}=$$ $$f(a-(t-t'), b-(t-t'))e'_{a+b+t'}.$$ For a given $t'$ we define $g(a,b):=f(a-(t-t'), b-(t-t'))$. Now we shall check the identity  (\ref{e1}) for $g$. For the elements $a'=a+t'-t, \ b'=b+t'-t, \ c'=c+t'-t$ we have $$f(b',c')f(a', b'+c'+t)=f(a', b')f(a'+b'+t, c')-f(a', c')f(a'+c'+t, b'),$$ consequently  $$g(b,c)g(a, b+c+t')=g(a, b)g(a+b+t', c)-g(a, c)g(a+c+t', b).$$ Hence, $g\in F_{t'}$ and
$L(f, t)\cong L(g, t').$ \end{proof}

Let us present an example:

\begin{ex}\label{ex1} Consider the group $\mathbb{Z}_2=\{\overline{0}, \overline{1}\}$, then function $f$ is defined on  $\mathbb{Z}_2\times \mathbb{Z}_2$. The four values of function $f$ can be represented in the form of $2\times 2$ matrix,
i.e. by matrix $(f(a,b))_{a,b=\overline{0}, \overline{1}}$. It easy to see that
$$F_{\overline{0}}=\{(f(a,b))_{a,b=\overline{0}, \overline{1}} \ : \  f \ \mbox{is a solution of} \ (\ref{e1}) \ \mbox{for} \ t=\overline{0}\}=$$ $$ \left\{
\left(\begin{array}{cc}
0& 0\\
\alpha_{10}& 0
\end{array}
\right); \left(\begin{array}{cc}
0& -\alpha_{10}\\
\alpha_{10}& 0
\end{array}
\right); \left(\begin{array}{cc}
0& 0\\
0& \alpha_{11}
\end{array}
\right), \ \mbox{where} \ \ \alpha_{10},\alpha_{11}\in K\right\},$$\\[0.21mm]
$$F_{\overline{1}}=\{(f(a,b))_{a,b=\overline{0}, \overline{1}} \ : \  f \ \mbox{is a solution of} \ (\ref{e1}) \ \mbox{for} \ t=\overline{1}\}=$$ $$\left\{
\left(\begin{array}{cc}
0& \alpha_{01}\\
0& 0
\end{array}
\right); \left(\begin{array}{cc}
0& -\alpha_{10}\\
\alpha_{10}& 0
\end{array}
\right); \left(\begin{array}{cc}
\alpha_{00}& 0\\
0& 0
\end{array}
\right), \ \mbox{where} \ \ \alpha_{00},\alpha_{01}, \alpha_{10}, \in K\right\}.$$
Consequently, the construction gives 6 Leibniz algebras $L_{\overline{i},j}(\theta, t)$, $i=0,1$, $j=1,2,3$, $\theta=\alpha_{00},\alpha_{01},\alpha_{11}$, $t=\overline{0},\overline{1}$. Moreover, one can check that $L_{\overline{0},i}(\theta, t)\cong L_{\overline{1},i}(\theta, t')$.
\end{ex}

\section{Periodic algebras}

Let $\widehat{G}$ be a subgroup of an additive group $G$, then the group $G$ is decomposable into cosets with respect to this subgroup: $G/\widehat{G}=\{g_0+\widehat{G}, g_1+\widehat{G}, \dots, g_{n-1}+\widehat{G}\},$ where $n$ is index of the subgroup $\widehat{G}$ in $G$.

Let $\widehat{G}$ be a subgroup of index $n$ for $G$. Then $G/\widehat{G}=\{G_0, \dots, G_{n-1}\}$, where $G_0=\widehat{G}$.

\begin{defn} The function $f : G\times G \rightarrow K$ is called $\widehat{G}$-periodic, if $f(a,b)=\alpha_{ij}$ for all $a\in G_i, \ b\in G_j.$
\end{defn}
In other words the function $f$ is periodic if its value at the point $(a,b)\in G\times G$ does not depend on $a$ and $b$, but the value depends only on the cosets to which $a$ and $b$ belong.
\begin{defn} The algebra $A_G(f,t)$ is called $\widehat{G}$-periodic, if it corresponds to a $\widehat{G}$-periodic function $f$.
\end{defn}

Let $\widehat{G}$ be a subgroup of an additive group $G$ and $A_{\widehat G}(f,0)$ be a $\widehat G$-periodic algebra. For the factor group $G/{\widehat G}=\{\widehat G, g_1+\widehat G, \dots\}$ we set $A_{g_i}=\{\alpha e_{g_i+h}: h\in\widehat G, \alpha\in K \}$. Then with respect to these sets $A_{\widehat G}(f,0)$ is a graded algebra.

If $a\in G_i, \ b\in G_j$ then instate $f(a,b)=\alpha_{ij}$ we write $f(a,b)=\alpha_{\overline{a}, \overline{b}}$ i.e. $\alpha_{ij}= \alpha_{\overline{a}, \overline{b}}$. In other words $\overline{a}$ denotes the number of the coset where belongs $a$.

\begin{thm}\label{t1} Let $\widehat{G}$ be a subgroup of index $n\geq 1$, a $\widehat{G}$-periodic  algebra $A_G(f,t)$ is right nilpotent if and only if
$$\alpha_{\overline{a}_1,\overline{a}_2}\alpha_{\overline{a_1+a_2},\overline{a}_3}\dots \alpha_{\overline{a_1+a_2+\dots+a_{k-1}},\overline{a}_k}=0,$$
for any $k$, $k\leq n$ and
for arbitrary $a_1,\dots,a_k\in G$ with $a_2+\dots+a_k\in \widehat{G}$.
\end{thm}
\begin{proof} It is known that the algebra $A$ is right nilpotent iff there exists $k$ such that $R_{x_2}R_{x_3}\dots R_{x_k}=0$ for arbitrary $x_2,\dots,x_k\in A$.
It is enough to check this condition for $x_2=e_{a_2},\dots,x_k=e_{a_k}$.

{\it Necessity.} Assume $A$ is right nilpotent.
We have
\begin{equation}\label{n1}
R_{e_{a_k}}R_{e_{a_{k-1}}}\dots R_{e_{a_2}}(e_{a_1})=\alpha_{\overline{a}_1,\overline{a}_2}\alpha_{\overline{a_1+a_2},\overline{a}_3}\dots \alpha_{\overline{a_1+a_2+\dots+a_{k-1}},\overline{a}_k}
e_{a_1+a_2+\dots+a_{k}+(k-1)t}=0.
\end{equation}

If $\overline{a_1+a_2+\dots+a_{k}}=\overline{a_1}$ and $$\alpha_{\overline{a}_1,\overline{a}_2}\alpha_{\overline{a_1+a_2},\overline{a}_3}\dots \alpha_{\overline{a_1+a_2+\dots+a_{k-1}},\overline{a}_k}\ne 0$$ then we can
consider  $(R_{e_{a_k}}R_{e_{a_{k-1}}}\dots R_{e_{a_2}}(e_{a_1}))^m$ which is not zero for any $m\geq 1$. So the condition of the theorem is necessary.

{\it Sufficiency.} Assume $$\alpha_{\overline{a}_1,\overline{a}_2}\alpha_{\overline{a_1+a_2},\overline{a}_3}\dots \alpha_{\overline{a_1+a_2+\dots+a_{k-1}},\overline{a}_k}=0,$$ for any $k$, $k\leq n$ and
for arbitrary $a_1,\dots,a_k\in G$ with $a_2+\dots+a_k\in \widehat{G}$. We shall prove that $A_G(f,t)$ is right nilpotent. Take $k>n$, then at least two of the following numbers coincide: $$\overline{a_1+a_2},\overline{a_1+a_2+a_3},\dots, \overline{a_1+\dots+a_{k+1}}.$$
Let $\overline{a_1+\dots+a_{p}}=\overline{a_1+\dots+a_{q}}$, $1\leq p-q\leq n$, i.e. $a_{p+1}+\dots+a_{q}\in \widehat{G}$. From this condition we have
$$\alpha_{\overline{a_1+\dots+a_p},\overline{a}_{p+1}}\alpha_{\overline{a_1+\dots+a_{p+1}},
\overline{a}_{p+2}}\dots \alpha_{\overline{a_1+a_2+\dots+a_{q-1}},\overline{a}_q}=0.$$
This implies that $R_{e_{a_k}}R_{e_{a_{k-1}}}\dots R_{e_{a_2}}(e_{a_1})=0$ for any $k>n$.
\end{proof}

\begin{pro}\label{pp} Assume $H, J$ are subgroups of $G$ with finite indexes such that $|G:J|$ divides  $|G:H|$ (where $| \ |$ stands for order). Then any $J$-periodic  algebra is $H$-periodic, but there are $H$-periodic  algebras which are not $J$-periodic.
\end{pro}
\begin{proof} Any $J$-periodic  algebra is represented by an $n\times n$ matrix $A_n=(\alpha_{ij})_{i,j=0,\dots,n-1}$, where $n=|G:J|$.  Any $H$-periodic  algebra is given by an $nm\times nm$ matrix $B_{nm}=(\beta_{ij})_{i,j=0,\dots,nm-1}$, where $nm=|G:H|$.  It is easy to see that a given $J$-periodic  algebra is $H$-periodic with $\beta_{nk+i, nl+j}=\alpha_{i,j}$, $i,j=0,\dots,n-1$, $k,l=0,\dots,m-1$. Consider an $H$-periodic algebra with $\beta_{n,n+1}\ne \alpha_{0,1}$, then this algebra is not $J$-periodic.
\end{proof}

Denote
 $$F_{t, \widehat{G}}^{per}=\{f : \ f \ {\rm is} \ \widehat{G}-\mbox{periodic and satisfies (\ref{e1})}\}.$$

\begin{pro}\label{p2} For any $t\in {\widehat G}, \ f\in F_{t,\widehat{G}}^{per}$ and $t'\in G_i,\ 1\leq i\leq n-1$ there exists $g\in F_{t',\widehat{G}}^{per}$ such that $L(f, t)\cong L(g, t').$
\end{pro}
\begin{proof}

Note that $t-t' \in G_i$. Since $f$ is periodic, we have $f(a-(t-t'),b-(t-t'))=\lambda_{jk}$ for $a\in t'+G_{j}, \ b\in t'+G_{k}.$
       We define $g(a,b)=f(a-(t-t'),b-(t-t'))=\lambda_{jk}$. It is easy to see that the function $g$ satisfies equation $(\ref{e1})$. The periodicity of $g$ follows from the periodicity of $f$. \end{proof}

\begin{rk}\label{r1} Using $f(a+t, b+t)=f(a,b)$ for all $t\in {\widehat G}$ one can show that for any $f\in F_{0,\widehat{G}}^{per}$ and $t\in \widehat G$ the $\widehat G$-periodic  Leibniz algebra $L(f, 0)$ is isomorphic to the $\widehat G$-periodic algebra $L(f, t)$.
\end{rk}

Now we shall study the isomorphic character of the following algebras

$$L(f,t): \ e_ae_b=f(a,b)e_{a+b+t} \ \ \mbox{and} \ \ L(g,t): \ e_ae_b=g(a,b)e_{a+b+t}.$$

Let $\varphi$ is an isomorphism and it is given by a matrix $(\gamma_{cd}),$ i.e.
$\varphi(e_c)=\sum_{d\in G}\gamma_{cd}e_d$. Then
$$\varphi(e_a)\varphi(e_b)=(\sum_{d\in G}\gamma_{ad}e_d)(\sum_{l\in G}\gamma_{bl}e_l)=\sum_{d,l\in G}\gamma_{ad}\gamma_{bl}f(d,l)e_{d+l+t}.$$

On the other hand, we have

$$\varphi(e_a)\varphi(e_b)=g(a,b)\varphi(e_{a+b+t})=g(a,b)\sum_{m\in G}\gamma_{a+b+t,m}e_m.$$
Comparing the coefficients of the basis elements for any $m$ we get

\begin{equation}\label{e2}
g(a,b)\gamma_{a+b+t,m}=\sum_{d\in G}\gamma_{a,d}\gamma_{b,m-d-t}f(d, m-d-t).
\end{equation}

In the periodic case we reduced the problem to the case $t=0$. Moreover, in this case the group and some subgroup of it are given. Thus for any $k$ the equality (\ref{e2}) with $f(i,j)=\alpha_{\overline{i}\overline{j}}, \ g(i,j)=\beta_{\overline{i}\overline{j}}$ has the following form:

\begin{equation}\label{e3}
\beta_{\overline{i},\overline{j}}\gamma_{i+j,k}=\sum_{s\in \mathbb{Z}}\gamma_{i,s}\gamma_{j,k-s}\alpha_{\overline{s}, \overline{k-s}}.
\end{equation}

Assume $f(a,b)=\alpha_{i,j}$ for $a\in g_i+\widehat{G}, \ b\in g_j+\widehat{G}.$

If $t\in g_s+\widehat{G}$, then by (\ref{e1}) we get:

1) If $a,b, c\in g_i+\widehat{G}$ then
$\alpha_{ii}\alpha_{i,\overline{2i+s}}=0.$

2) If $a,b\in g_i+\widehat{G}, \ c\in g_j+\widehat{G}$ then
$\alpha_{ij}\alpha_{i,\overline{i+j+s}}=\alpha_{ii}\alpha_{\overline{2i+s},j} -
\alpha_{ij}\alpha_{\overline{i+j+s},i}.$

3) If $a,c\in g_i+\widehat{G}, \ b\in g_j+\widehat{G}$ then
$\alpha_{ji}\alpha_{i,\overline{i+j+s}}=\alpha_{ij}\alpha_{\overline{i+j+s},i} -
\alpha_{ii}\alpha_{\overline{2i+s},j}.$

4) If $b,c\in g_i+\widehat{G}, \ a\in g_j+\widehat{G}$ then
$\alpha_{ii}\alpha_{j,\overline{2i+s}}=0.$

5) If $a\in g_i+\widehat{G}, \ b\in g_j+\widehat{G}, \ c\in g_k+\widehat{G}$ then
$\alpha_{jk}\alpha_{i,\overline{j+k+s}}=\alpha_{ij}\alpha_{\overline{i+j+s},k} -
\alpha_{ik}\alpha_{\overline{i+k+s},j}.$

\begin{pro}\label{ppl} Assume $H, J $ are subgroups of $G$ with finite indexes such that $|G:J|$ divides  $|G:H|$. Let $G/J=\{G_0,\dots,G_{n-1}\}$ be the factor group with $G_i+G_j=G_{i+j({\rm mod} n)}$, $G/H=\{H_0,\dots,H_{nm-1}\}$ be the factor group with $H_i+H_j=H_{i+j({\rm mod} nm)}$. Then any $J$-periodic  Leibniz algebra $L$ is $H$-periodic.
\end{pro}
\begin{proof} Let $L(f,t)$ be a $J$-periodic Leibniz algebra which corresponds to the matrix $A_n=(\alpha_{ij})_{i,j=0,\dots,n-1}$. Using notions of the proof of Proposition \ref{pp} and also taking $\beta_{nk+i, nl+j}=\alpha_{i,j}$, $i,j=0,\dots,n-1$; $k,l=0,\dots,m-1,$ we have that the corresponding algebra $L(f,t)$ is $H$-periodic. It remains to show that $L(f,t)$ satisfies the Leibniz identity with respect to the $\beta_{ij}$ which were defined above. Let $G/H=\{H_0,\dots,H_{nm-1}\}$ be the corresponding factor group.
Taking $a\in H_i$, $b\in H_j$, $c\in H_p$ and $t\in H_q$ then using the condition of the proposition, from 1)-5) mentioned above we get
\begin{equation}\label{e5}
    \beta_{j,p}\beta_{i,j+p+q({\rm mod}nm)} = \beta_{i,j}\beta_{i+j+q({\rm mod}nm),p}- \beta_{i,p}\beta_{i+p+q ({\rm mod}nm),j}.
\end{equation}
Assume $i=nn_1+i', j=nn_2+j', p=nn_3+p', q=nn_4+q'$, with $i',j',p',q'\in\{0,\dots,n-1\}$ then (\ref{e5}) becomes:
\begin{equation}\label{e6}
    \alpha_{j',p'}\alpha_{i',j'+p'+q'({\rm mod}n)} = \alpha_{i',j'}\alpha_{i'+j'+q'({\rm mod}n),p'}-\alpha_{i',p'}\alpha_{i'+p'+q'({\rm mod}n),j'},
\end{equation}
which holds since $L(f,t)$ is a Leibniz algebra with respect to $\alpha_{ij}$.
 \end{proof}

Let us give some examples for our construction:
\subsection{$2\mathbb{Z}$-periodic Leibniz algebras} Consider the case $G=\mathbb{Z}, \ \widehat{G}=2\mathbb{Z}$.

In this case $\mathbb{Z}/2\mathbb{Z}=\{\overline{0}, \overline{1}\}.$

Let $t\in \overline{0}$. We consider all possible cases for $a,b,c\in\{\overline{0},\overline{1}\}:$

1) $a,b,c \in \overline{0} \Rightarrow \alpha_{00}=0.$

2) $a,b \in \overline{0}, \ c \in \overline{1} \Rightarrow \alpha_{01}^2=\alpha_{00}\alpha_{01}-\alpha_{01}\alpha_{10}.$

3) $a,c \in \overline{0}, \ b \in \overline{1} \Rightarrow \alpha_{10}\alpha_{01}=\alpha_{01}\alpha_{10}-\alpha_{00}\alpha_{01}.$

4) $b,c \in \overline{0}, \ a \in \overline{1} \Rightarrow \alpha_{00}\alpha_{10}=0.$

5) $a,b, c \in \overline{1} \Rightarrow \alpha_{11}\alpha_{10}=0.$

6) $a,b \in \overline{1}, \ c \in \overline{0} \Rightarrow \alpha_{10}\alpha_{11}=\alpha_{11}\alpha_{00}-\alpha_{10}\alpha_{11}.$

7) $a,c \in \overline{1}, \ b \in \overline{0} \Rightarrow \alpha_{01}\alpha_{11}=\alpha_{10}\alpha_{11}-\alpha_{11}\alpha_{00}.$

8) $b,c \in \overline{1}, \ a \in \overline{0} \Rightarrow \alpha_{11}\alpha_{00}=0.$

After simplifications we get

$$\left\{\begin{array}{ll}
\alpha_{00}=0&\\
\alpha_{01}(\alpha_{01}+\alpha_{10})=0&\\
\alpha_{11}\alpha_{10}=0&\\
\alpha_{01}\alpha_{11}=0.&\\
\end{array}
\right.$$

Consider the all possible cases:

{\bf Case 1.} $\alpha_{11}=0.$ Then $\alpha_{01}(\alpha_{01}+\alpha_{10})=0.$

{\bf Case 1.1.} $\alpha_{01}=0.$ Then $\alpha_{10}$ is an arbitrary parameter.

{\bf Case 1.2.} $\alpha_{01}\neq 0.$ Then $\alpha_{01}=-\alpha_{10}.$

{\bf Case 2.} $\alpha_{11}\neq0.$ Then $\alpha_{01}=\alpha_{10}=0.$

Thus for $t\in \overline{0}$ we obtain the following matrices for the structural constants:

$$\left(\begin{array}{cc}
0& 0\\
\alpha_{10}& 0
\end{array}
\right); \left(\begin{array}{cc}
0& -\alpha_{10}\ne 0\\
\alpha_{10}& 0
\end{array}
\right); \left(\begin{array}{cc}
0& 0\\
0& \alpha_{11}\ne 0
\end{array}
\right).$$

By scaling the basis we obtain the following three $2\mathbb Z$-periodic Leibniz algebras:

$$L_{\overline{0}1}(\alpha,t):
e_{2k-1}e_{2m}=\alpha e_{2(k+m)-1+t}, \ \ \alpha\in \{0,1\};$$
$$L_{\overline{0}2}(t): \ \left\{\begin{array}{ll}
e_{2k}e_{2m-1}=- e_{2(k+m)-1+t},\\
e_{2k-1}e_{2m}= e_{2(k+m)-1+t},\\
\end{array}
\right.;$$
$$L_{\overline{0}3}(t): \
e_{2k-1}e_{2m-1}= e_{2(k+m-1)+t}.$$

Similarly for $t\in \overline{1}$ we obtain the following periodic Leibniz algebras:

$$L_{\overline{1}1}(\alpha,t): \ e_{2k}e_{2m-1}=\alpha e_{2(k+m)-1+t}, \alpha\in \{0,1\}.$$
$$L_{\overline{1}2}(t): \ \left\{\begin{array}{ll}
e_{2k}e_{2m-1}=- e_{2(k+m)-1+t}\\
e_{2k-1}e_{2m}= e_{2(k+m)-1+t},\\
\end{array}
\right.$$
$$L_{\overline{1}3}(t): \
e_{2k}e_{2m}=e_{2(k+m)+t}.$$

\begin{pro} For any $k=1,2,3$ we have $L_{\overline{0}k}(\theta,t)\cong L_{\overline{1}k}(\theta,t')$. Moreover the algebras $L_{\overline{0}1}, L_{\overline{0}2}, L_{\overline{0}3}$ are pairwise non isomorphic.
\end{pro}
\begin{proof}
Take the isomorphism $\varphi : L_{\overline{0}k}(\theta,t) \longrightarrow L_{\overline{1}k}(\theta,t')$ defined by $\varphi(e_i)=e'_{i+t-t'}$. Then the proof can be completed by verifying algebraic property of isomorphism. The statement that the algebras $L_{\overline{0}1}, L_{\overline{0}2}, L_{\overline{0}3}$ are pairwise non isomorphic follows from the property of the algebra to be Lie or commutative algebra, i.e., property to satisfy the different identities.
\end{proof}

\subsection{$3\mathbb{Z}$-periodic Leibniz algebras}.\\

In the case where $G=\mathbb{Z}, \ \widehat{G}=3\mathbb{Z}$ and $a\in \overline{k}, b\in \overline{i}, c\in \overline{j}, t\in \overline{s}$ we have:

$$\alpha_{ij}\alpha_{k,\overline{i+j+s}}=-\alpha_{ji}\alpha_{k,\overline{i+j+s}}.$$

Consider $\mathbb{Z}/3\mathbb{Z}=\{\overline{0}, \overline{1}, \overline{2}\}.$

For $t\in \overline{0}$ from (\ref{e1}) we get the following

1) $a,b,c \in \overline{0} \Rightarrow \alpha_{00}=0.$

2) $a,b \in \overline{0}, \ c \in \overline{1} \Rightarrow \alpha_{01}^2=\alpha_{00}\alpha_{01}-\alpha_{01}\alpha_{10}.$

3) $a,c \in \overline{0}, \ b \in \overline{1} \Rightarrow \alpha_{10}\alpha_{01}=\alpha_{01}\alpha_{10}-\alpha_{00}\alpha_{01}.$

4) $b,c \in \overline{0}, \ a \in \overline{1} \Rightarrow \alpha_{00}\alpha_{10}=0.$

5) $a,b, c \in \overline{1} \Rightarrow \alpha_{11}\alpha_{12}=0.$

6) $a,b \in \overline{1}, \ c \in \overline{0} \Rightarrow \alpha_{10}\alpha_{11}=\alpha_{11}\alpha_{20}-\alpha_{10}\alpha_{11}.$

7) $a,c \in \overline{1}, \ b \in \overline{0} \Rightarrow \alpha_{01}\alpha_{11}=\alpha_{10}\alpha_{11}-\alpha_{11}\alpha_{20}.$

8) $b,c \in \overline{1}, \ a \in \overline{0} \Rightarrow \alpha_{11}\alpha_{02}=0.$

1') $a,b \in \overline{0}, \ c \in \overline{2} \Rightarrow \alpha_{02}^2=\alpha_{00}\alpha_{02}-\alpha_{02}\alpha_{20}.$

2') $a,c \in \overline{0}, \ b \in \overline{2} \Rightarrow \alpha_{20}\alpha_{02}=\alpha_{02}\alpha_{20}-\alpha_{00}\alpha_{02}.$

3') $b,c \in \overline{0}, \ a \in \overline{2} \Rightarrow \alpha_{00}\alpha_{20}=0.$

4') $a,b, c \in \overline{2} \Rightarrow \alpha_{22}\alpha_{21}=0.$

5') $a,b \in \overline{2}, \ c \in \overline{0} \Rightarrow \alpha_{20}\alpha_{22}=\alpha_{22}\alpha_{10}-\alpha_{20}\alpha_{22}.$

6') $a,c \in \overline{2}, \ b \in \overline{0} \Rightarrow \alpha_{02}\alpha_{22}=-\alpha_{20}\alpha_{22}.$

7') $b,c \in \overline{2}, \ a \in \overline{0} \Rightarrow \alpha_{22}\alpha_{01}=0.$

1'') $a,b \in \overline{1}, \ c \in \overline{2} \Rightarrow \alpha_{12}\alpha_{10}=\alpha_{11}\alpha_{22}-\alpha_{12}\alpha_{01}.$

2'') $a,c \in \overline{1}, \ b \in \overline{2} \Rightarrow \alpha_{21}\alpha_{10}=-\alpha_{12}\alpha_{10}.$

3'') $b,c \in \overline{1}, \ a \in \overline{2} \Rightarrow \alpha_{11}\alpha_{22}=0.$

4'') $a,b, c \in \overline{2} \Rightarrow \alpha_{22}\alpha_{21}=0.$

5'') $a,b \in \overline{2}, \ c \in \overline{1} \Rightarrow \alpha_{21}\alpha_{20}=\alpha_{22}\alpha_{11}-\alpha_{21}\alpha_{02}.$

6'') $a,c \in \overline{2}, \ b \in \overline{1} \Rightarrow \alpha_{12}\alpha_{20}=-\alpha_{21}\alpha_{20}.$

7'') $b,c \in \overline{2}, \ a \in \overline{1} \Rightarrow \alpha_{22}\alpha_{11}=0.$

a) $a \in \overline{0}, \ b \in \overline{1}, \ c \in \overline{2} \Rightarrow \alpha_{12}\alpha_{00}=\alpha_{01}\alpha_{12}-\alpha_{02}\alpha_{21}.$

a') $a \in \overline{0}, \ b \in \overline{2}, \ c \in \overline{1} \Rightarrow \alpha_{21}\alpha_{00}=-\alpha_{12}\alpha_{00}.$

b) $a \in \overline{1}, \ b \in \overline{0}, \ c \in \overline{2} \Rightarrow \alpha_{02}\alpha_{12}=\alpha_{10}\alpha_{12}-\alpha_{12}\alpha_{00}.$

b') $a \in \overline{1}, \ b \in \overline{2}, \ c \in \overline{0} \Rightarrow \alpha_{20}\alpha_{12}=-\alpha_{02}\alpha_{12}.$

c) $a \in \overline{2}, \ b \in \overline{0}, \ c \in \overline{1} \Rightarrow \alpha_{01}\alpha_{21}=\alpha_{20}\alpha_{21}-\alpha_{21}\alpha_{00}.$

c') $a \in \overline{2}, \ b \in \overline{1}, \ c \in \overline{0} \Rightarrow \alpha_{10}\alpha_{21}=-\alpha_{01}\alpha_{21}.$

Solving these equations we get the following eleven matrices:
$$A_1: \ \left(\begin{array}{ccc}
0& 0&0\\
\alpha_{10} & 0&0\\
\alpha_{20}&0&0\\
\end{array}
\right); \ A_2\ \left(\begin{array}{ccc}
0& 0&0\\
0 & 0&0\\
0&\alpha_{21}\neq0&0\\
\end{array}
\right); \ A_3\ \left(\begin{array}{ccc}
0& 0&0\\
0 & 0&\alpha_{12}\neq0\\
0&\alpha_{21}&0\\
\end{array}
\right);$$

$$A_4\left(\begin{array}{ccc}
0& 0&\alpha_{02}\neq0\\
\alpha_{10} & 0&0\\
-\alpha_{02}&0&0\\
\end{array}
\right); \ A_5\  \left(\begin{array}{ccc}
0& \alpha_{01}\neq0&0\\
-\alpha_{01} & 0&0\\
\alpha_{20}&0&0\\
\end{array}
\right);$$

$$ A_6\ \left(\begin{array}{ccc}
0& \alpha_{01}\neq0&\alpha_{02}\neq0\\
-\alpha_{01} & 0&0\\
-\alpha_{02}&0&0\\
\end{array}
\right);\ \
A_7\left(\begin{array}{ccc}
0& \alpha_{01}\neq0&-\alpha_{01}\\
-\alpha_{01} & 0&\alpha_{12}\neq0\\
\alpha_{01}&-\alpha_{12}&0\\
\end{array}
\right); $$ $$  A_8\ \left(\begin{array}{ccc}
0& 0&0\\
0 & 0&\alpha_{12}\\
0&0&\alpha_{22}\neq0\\
\end{array}
\right); \ A_9\ \left(\begin{array}{ccc}
0& 0&-\alpha_{20}\\
2\alpha_{20} & 0&0\\
\alpha_{20}\neq0&0&\alpha_{22}\neq0\\
\end{array}
\right);$$

$$A_{10}\left(\begin{array}{ccc}
0& 0&0\\
0 & \alpha_{11}\neq0&0\\
0&\alpha_{21}&0\\
\end{array}
\right); \ A_{11}\ \left(\begin{array}{ccc}
0& \alpha_{01}\neq0&0\\
-\alpha_{01} & \alpha_{11}\neq0&0\\
-2\alpha_{01}&0&0\\
\end{array}
\right).$$

Now by an appropriate scaling of basis we can take all non-zero parameters to be equal to 1.  Thus we have proved the following:

\begin{pro}\label{p3} For $t=0$, there are the following eleven (infinite dimensional) $3\mathbb{Z}$-periodic complex  Leibniz algebras (all omitted
products are to be understood as being equal to zero):

$$L_1(\alpha,\beta):
\left\{\begin{array}{ll}
e_{3k+1}e_{3m}=\alpha e_{3(k+m)+1}\\
e_{3k+2}e_{3m}=\beta e_{3(k+m)+2}\\
\end{array}
\right.;\ \ L_2:
\ \ e_{3k+2}e_{3m+1}=e_{3(k+m+1)};$$
$$L_3(\beta):
\left\{\begin{array}{ll}
e_{3k+1}e_{3m+2}=e_{3(k+m+1)}\\
e_{3k+2}e_{3m+1}=\beta e_{3(k+m+1)}\\
\end{array}
\right.; \ \ L_4(\beta):
\left\{\begin{array}{lll}
e_{3k}e_{3m+2}=e_{3(k+m)+2}\\
e_{3k+2}e_{3m}=-e_{3(k+m)+2},\\
e_{3k+1}e_{3m}=\beta e_{3(k+m)+1}\\
\end{array}
\right.; $$
$$L_5(\beta):
\left\{\begin{array}{lll}
e_{3k}e_{3m+1}=e_{3(k+m)+1}\\
e_{3k+1}e_{3m}=-e_{3(k+m)+1}\\
e_{3k+2}e_{3m}=\beta e_{3(k+m)+2}\\
\end{array}
\right.; \ \ L_6(\beta):
\left\{\begin{array}{llll}
e_{3k}e_{3m+1}=e_{3(k+m)+1}\\
e_{3k+1}e_{3m}=-e_{3(k+m)+1}\\
e_{3k}e_{3m+2}=\beta e_{3(k+m)+2}\\
e_{3k+2}e_{3m}=-\beta e_{3(k+m)+2},\\
\end{array}
\right. \beta\ne 0;$$
$$L_7:
\left\{\begin{array}{llllll}
e_{3k}e_{3m+1}= e_{3(k+m)+1}\\
e_{3k+1}e_{3m}=- e_{3(k+m)+1},\\
e_{3k}e_{3m+2}=- e_{3(k+m)+2}\\
e_{3k+2}e_{3m}= e_{3(k+m)+2}\\
e_{3k+1}e_{3m+2}= e_{3(k+m+1)}\\
e_{3k+2}e_{3m+1}=- e_{3(k+m+1)}\\
\end{array}
\right.; \ \ L_8(\alpha):
\left\{\begin{array}{ll}
e_{3k+1}e_{3m+2}=\alpha e_{3(k+m+1)}\\
e_{3k+2}e_{3m+2}=e_{3(k+m+1)+1},\\
\end{array}
\right.; $$
$$L_9:
\left\{\begin{array}{llll}
e_{3k+2}e_{3m}= e_{3(k+m)+2}\\
e_{3k}e_{3m+2}=- e_{3(k+m)+2},\\
e_{3k+1}e_{3m}=2 e_{3(k+m)+1}\\
e_{3k+2}e_{3m+2}= e_{3(k+m+1)+1}\\
\end{array}
\right.;\ \ L_{10}(\beta):
\left\{\begin{array}{ll}
e_{3k+1}e_{3m+1}= e_{3(k+m)+2}\\
e_{3k+2}e_{3m+1}=\beta e_{3(k+m+1)},\\
\end{array}
\right.;$$
$$L_{11}:
\left\{\begin{array}{llll}
e_{3k}e_{3m+1}= e_{3(k+m)+1}\\
e_{3k+1}e_{3m}=- e_{3(k+m)+1},\\
e_{3k+2}e_{3m}=-2 e_{3(k+m)+2}\\
e_{3k+1}e_{3m+1}= e_{3(k+m)+2}\\
\end{array}
\right. \ \  {\rm where} \ \  k,m\in \mathbb{Z}.$$
\end{pro}

As a corollary of Theorem \ref{t1} we have

\begin{cor}\label{c1} A $3\mathbb{Z}$-periodic  Leibniz algebra $L$ is right nilpotent if and only if $\alpha_{i0}=0$, $\forall i=0,1,2$, $\alpha_{01}\alpha_{11}\alpha_{21}= \alpha_{02}\alpha_{12}\alpha_{22}=0$ and $\alpha_{01}\alpha_{12}=\alpha_{11}\alpha_{22}=\alpha_{21}\alpha_{02}=0$.
\end{cor}
This corollary with Proposition \ref{p3} gives the following

\begin{cor}\label{c2} The algebras $L_2, L_3, L_8, L_{10}$ mentioned in Proposition \ref{p3}  are  nilpotent.
\end{cor}

\subsection{$n\mathbb{Z}$-periodic algebras}

Denote $N_n=\{0,\dots,n-1\}$.
The following proposition gives some properties of $n\mathbb{Z}$-periodic complex  Leibniz algebras:

\begin{pro}\label{p4} Let $L$ be an $n\mathbb{Z}$-periodic  Leibniz algebra. Then

1) $L^2=L$ if and only if for any $p\in N_n$ there exists $i\in N_n$ such that
$\alpha_{i,\overline{p-i}}\ne 0$;

2) For any $n\mathbb{Z}$-periodic  Leibniz algebra $L$ there exists an algebra $L'$ such that $\alpha_{00}'=0$ (i.e. $e_{nk}e_{nm}=0, \ k,m\in \mathbb{Z}$) and $L\cong L'$;

3) $\{e_{ns+j} | \ s\in\mathbb{Z}\}\subseteq Ann_r(L)$ ($\{e_{ns+j} | \ s\in\mathbb{Z}\}\subseteq Ann_l(L)$) if and only if the $j$-th column (row) of the matrix $A$ of structural constants of algebra $L$ is zero;

4) The element $e_{nk+i}, \ i\neq0$ is right (left) nilpotent if and only if in $i$-th column (row) of $A$ there exists a zero;

5) The algebra $L$ is a Lie algebra iff $A_n=(\alpha_{ij})_{i,j=0,\dots,n-1}$ is a skew-symmetric matrix.
\end{pro}
\begin{proof}

1) {\it Necessity.} If $L^2=L$ then for any $e_c\in L$ there are $a,b\in \mathbb{Z}$ such that $\alpha_{\overline{a},\overline{b}}\ne 0$ and
$$e_c=e_ae_b=\alpha_{\overline{a},\overline{b}}e_{a+b}.$$
If $a\in \overline{i}$, $b\in \overline{j}$ and $c\in \overline{p}$ then from $c=a+b$ we get
$j=\overline{p-i}$.

{\it Sufficiency.} Assume for any $p\in N_n$ there exists $i\in N_n$ such that
$\alpha_{i,\overline{p-i}}\ne 0$. Define $a=i$ and $b=\overline{p-i}$ then
$$e_ae_b=\alpha_{i,\overline{p-i}}e_{a+b}=\alpha_{i,\overline{p-i}}e_c.$$
Hence $e_c\in L^2$ and $L^2=L$.

2) Follows from Proposition \ref{p2} and Remark \ref{r1}.

3) Straightforward.

4) This follows from the following equality and the property of the complete residue system

$$e_{nk+i}^{(n+1)_r}=\alpha_{i,i}\alpha_{\overline{2i},i}\dots \alpha_{\overline{ni},i}e_{n(nk+k+i)+i}=\alpha_{0,i}\alpha_{1,i}\dots \alpha_{n-1,i}e_{n(nk+k+i)+i}.$$

5) It is easy to see that a  Leibniz algebra $L(f,t)$ is a Lie algebra iff its matrix $A$ of structural constants is skew-symmetric.
The matrix $A$ of a $n\mathbb{Z}$-periodic  algebra $L$ is constructed by blocs $A_n$. From the following equalities it follows that $A$ is skew-symmetric iff $A_n$ is skew-symmetric:
$$\alpha_{nk+i,nl+j}=\alpha_{ij}=-\alpha_{ji}=-\alpha_{nl+j,nk+i}.$$ This completes the proof.  \end{proof}

The following proposition gives a characterization of subalgebras.

\begin{pro}\label{psa} Let $L(f,0)$ be an $n\mathbb Z$-periodic algebra with matrix $A_n=(\alpha_{ij})_{i,j=0,\dots,n-1}$ and let $s\in \{1,\dots,n-1\}$. If
\begin{equation}\label{b}
(\alpha_{is}, \alpha_{si})\ne (0,0), \ \ \mbox{for\ \ all} \ \ i=0,1,\dots, n-1,
\end{equation}
 then
 $\langle e_{nk+s}: k\in \mathbb Z\rangle=L(f,0)$.
\end{pro}
\begin{proof} From $e_{nk+s}e_{nm+s}=\alpha_{s,s} e_{n(k+m)+2s}$, by the condition $\alpha_{s,s}\ne 0$ we get $e_{n(k+m)+2s}\in \langle e_{nk+s}: k\in \mathbb Z\rangle$.
Similarly, we get that $e_{nq+ps}\in \langle e_{nk+s}: k\in \mathbb Z\rangle$, for any $q\in \mathbb Z$ and $p\in \mathbb N$. Thus we get that the complete residue system $\{e_{nk+i}, i=0,\dots,n-1; \, k\in \mathbb Z\}$, is contained in $\langle e_{nk+s}: k\in \mathbb Z\rangle=L(f,0)$.
\end{proof}

\begin{rk}
Note that if $s\ne 0$ and $\langle e_{nk+s}: k\in \mathbb Z\rangle\subseteq I=ideal\langle xx: x\in L\rangle$ and the condition (\ref{b}) is satisfied then $I=L$. Consequently, $L$ is an Abelian algebra and
$A_n=0,$ which is contradicts condition (\ref{b}). Thus from $\langle e_{nk+s}: k\in \mathbb Z\rangle\subseteq I$ it follows that there exists $i\in \{0,\dots,n-1\}$ such that
$\alpha_{is}=\alpha_{si}=0$.
\end{rk}

Consider the algebras $L_i$, $i=1,\dots,11$ listed in Proposition \ref{p3}.  Using Proposition \ref{ps} we get the following:

\begin{pro}\label{s} 1) Each of the algebras $L_1, L_2, L_3, L_4, L_5, L_6, L_8,L_{10}$ are solvable with index of solvability 2.

2) The algebra $L_7$ is not solvable.

3) The algebras $L_9, L_{11}$ are solvable with index of solvability 4.
\end{pro}

\begin{proof} 1) Solvability of $L_1-L_6, L_8, L_{10}$ obvious.

2) We shall check that $L_7$ is not solvable. It is easy to see that $L_7^2=L_7$, consequently
$L^{[k]}_7=L_7$ for all $k\in \mathbb N$.

3) We have $L_{11}^{[2]}=\langle (e_ie_j)(e_ke_l); i,j,k,l\in \mathbb Z\rangle$. Considering
all possible values of $i,j,k,l$ we obtain $L_{11}^{[2]}\subset\langle e_{3i}, e_{3j+2}; i,j\in \mathbb Z\rangle$. Since $\alpha_{00}=\alpha_{22}=0$, we get $L_{11}^{[3]}\subset\langle e_{3j+2}; j\in \mathbb Z\rangle$, now since $\alpha_{22}=0$ we obtain $L_{11}^{[4]}=0$.
By a similar argument one can show that $L_9$ is also solvable with index 4.
\end{proof}

\begin{thm}\label{t2}
The following infinite dimensional $3\mathbb Z$-periodic Leibniz algebras (mentioned in Proposition \ref{p3}) for fixed parameters $\alpha, \beta$ are pairwise non-isomorphic:
$$L_1(\alpha,\beta), \, L_3(\beta), \, L_4(\beta),\, L_6(\beta),\, L_7, \, L_8(\alpha), \, L_9, L_{11}.$$
\end{thm}
\begin{proof} Since $L_7$ is not solvable, it is not isomorphic to the algebras $L_i$, $i\ne 7$.
Since the algebras $L_9,L_{11}$ are solvable with index 4, each of them is not isomorphic to each of the algebras $L_1-L_6, L_8, L_{10}$. Moreover, since $L_9$ is not a Lie algebra, but $L_{11}$ is a Lie algebra, they are not isomorphic. We have the following:

The algebra $L_8$ is isomorphic to $L_{10}$ with the isomorphism:
$$e_{3k}'=e_{3(k+1)}, \,  e_{3k+1}'=e_{3k+2}, \,  e_{3k+2}'=e_{3(k+1)+1}, \,k\in \mathbb Z.$$

The algebra $L_2$ is isomorphic to $L_{3}(0)$ with the isomorphism:
$$e_{3k}'=e_{3k}, \,  e_{3k+1}'=e_{3k+2}, \,  e_{3k+2}'=e_{3k+1}, \,k\in \mathbb Z.$$
Moreover, $L_2\subseteq L_{3}(\beta)$.

The algebra $L_4$ is isomorphic to $L_{5}$ with the isomorphism:
$$e_{3k}'=e_{3k}, \,  e_{3k+1}'=e_{3k+2}, \,  e_{3k+2}'=e_{3k+1}, \,k\in \mathbb Z.$$

$L_4(\alpha)$ is not isomorphic to the algebra $L_6(\beta)$. Indeed, $L_6(\beta)$ is a Lie algebra but
$L_4(\alpha\ne 0)$ is not a Lie algebra. For $\alpha=0$ we have $\{0\}\ne {\rm Center}(L_4)=\{e_{3k+1}: k\in \mathbb Z\}$ but ${\rm Center}(L_6)=\{0\}$.
Similarly one can show that $L_1(\alpha,\beta)$ is not isomorphic to $L_6(\gamma)$.

$L_1(\alpha,\beta)$ is not isomorphic to $L_4(\gamma)$, because $L_4(0)$ is a Lie algebra but
$L_1(\alpha,\beta)$ is not a Lie algebra. For $\gamma\ne 0$ we have ${\rm Ann}_l(L_4)=\{0\}$ but ${\rm Ann}_l(L_1)=\{e_{3m}: m\in \mathbb Z\}$.

$L_3(\beta)$ is not isomorphic to $L_8(\alpha)$. This follows from the following relations
 $$L_3^2={\rm Center}(L_3),$$
$${\rm Center}(L_8(\alpha\ne 0))=\langle e_{3k}\rangle\subsetneqq L_8^2(\alpha\ne 0),$$
$${\rm Center}(L_8(0))=\langle e_{3k}, e_{3k+1}\rangle\supsetneqq L_8^2(0)=\langle e_{3k+1}\rangle.$$
\end{proof}

\section*{Acknowledgments}

The second and the third named authors would like to acknowledge
the hospitality of the "Institut f\"{u}r Angewandte Mathematik",
Universit\"{a}t Bonn (Germany). This work is supported in part by
the DFG AL 214/36-1 project (Germany).
We thank both referees for their helpful comments.

\end{document}